\documentclass[invmat]{svjour}
\usepackage{color}
\usepackage{amsmath}
\usepackage{verbatim}
\usepackage{graphicx}
\usepackage{psfrag}
\usepackage{mathtools}

\def\ds{\displaystyle}
\def\eps{{\varepsilon}}

\def\spt{{\rm spt}}

\def\bfR{\mbox{\boldmath$R$}}

\def\rhom{{\rho_m}}

\numberwithin{equation}{section}
\begin{document}
\title{Structure of the fundamental solution of a nonconvex conservation
law
\thanks{This work was supported by the Korea Science and Engineering Foundation(KOSEF) grant funded by the Korea government(MOST) (No.R01-2007-000-11307-0).}}

%%Run-in style author template:
\author{Yong-Jung Kim \and Young-Ran Lee\thanks{Corresponding Author}}

 \institute{Department of Mathematical Sciences, KAIST \\
291 Daehak-ro, Yuseong-gu, Daejeon 305-701, Korea\\
email: yongkim@kaist.edu\\
\\
\\
Department of Mathematics, Sogang University\\
35 Baekbeom-ro, Mapo-gu, Seoul 121-742, Korea.\\
email: younglee@sogang.ac.kr}

% etc
%
\titlerunning{Structure of the fundamental solution with nonconvex flux}
\date{\today}
%
%\date{Received: 2008-03-15 / Revised version: \today}
\maketitle

%\keywords{fundamental solution, source-type solution, N-wave,
%convex-concave envelope, entropy admissibility condition, weak
%solution, non-convexity}

\begin{abstract}
The structure of a signed fundamental solution of a conservation law is studied without the convexity assumption. The types of shocks and rarefaction waves are classified together with their interactions. A comprehensive picture of a global dynamics of a nonconvex flux is discussed in terms of characteristic maps and dynamical convex-concave envelopes.
\end{abstract}

\maketitle

\tableofcontents

\section{Introduction}\label{intro}

We consider a Cauchy problem of a scalar conservation law,
\begin{equation}\label{eqn}
\partial_t u+\partial_x f(u)=0,\quad u(x,0)=u_0(x),\quad x\in\bfR,\;t>0,
\end{equation}
where the flux $f$ is nonconvex and satisfies two hypotheses:
\begin{equation}\label{H1}
\begin{array}{c}\ds
\mbox{\it $f$ has a finite number of inflection points},\\
\ds
{f(u)\over u}\to\infty \mbox{~~~as~~~}u\to\infty.\\
\end{array}
\end{equation}
We assume a smooth flux $f\in C^1$ and, without loss of generality,
\begin{equation}\label{nomalization}
f(0)=f'(0)=0.
\end{equation}
The signed fundamental solution with mass $m>0$, denoted by $\rhom$ or simply by $\rho$, is the nonnegative entropy solution of (\ref{eqn}) that satisfies
\begin{equation}\label{IC-rho}
\lim_{t\to 0}\rho_m(x,t)=m\delta(x).
\end{equation}
This signed fundamental solution has been constructed in \cite{HaKim} using dynamical convex-concave envelopes. The purpose of this paper is to clarify the components of this fundamental solution and combine them to obtain the global dynamics of a nonconvex conservation law.

The fundamental solution plays the key role in the Cauchy problem. The convolution of the fundamental solution and the initial value is the solution for an autonomous linear problem. Even for a nonlinear problem the fundamental solution plays the key role. The Oleinik one-sided inequality for convex flux is written by
$$
f'(u)_x\le 1/t,\qquad t>0,
$$
and the fundamental solution satisfies the equality, $f'(\rhom)_x=1/t$.  In other words, this inequality, which gives the uniqueness and the sharp regularity to the conservation law, is basically a comparison to the fundamental solution. There have been several attempts to extend the inequality to nonconvex cases (see \cite{MR2405854,MR688972,MR1855004,MR2119939}). It is clear that the key of such an extension is in a better understanding of the fundamental solution. One may find an extension of such one-sided inequalities from the first author's recent work \cite{Kim14}. The asymptotic analysis is sometimes a study of a process how a solution turns into the shape of a fundamental solution eventually (see, e.g., \cite{MR621249,MR2127996,MR1987091,MR2373653}). This indicates that the fundamental solution reflects the intrinsic properties of PDEs.

The nonlinear scalar conservation law provides the shock wave theory in a simplest form. The behavior of a genuinely nonlinear flux or, equivalently, a convex flux is well understood (see \cite{MR0457947,MR0267257,MR0093653,MR0094541}). In particular, the Lax-Hopf transformation \cite{MR0093653} may make the solution even explicit. The main reason for this simplicity is that the information is destroyed along a shock, but never produced. A conservation law with a nonconvex flux has a quite different property that generates new information. Examples of nonconvex flux are from Buckley-Leverett, thin film and many others (see \cite{MR1725916,BuckleyLeverett,MR0328387}). The existence and the uniqueness of a bounded solution and the convergence of zero viscosity limit hold true for both convex and nonconvex cases (see \cite{MR0435615,MR715692,MR0267257,MR2001466}). The BV-boundedness holds for a uniformly convex case, but not for a general nonconvex case (see  \cite{MR690889,MR818862,MR785530,MR603391,MR0216338} for more regularity properties). The biggest difference is the complexity of the dynamics. Even though the structure of a solution has been studied for a case with a  single inflection point (see \cite{MR0435615,MR785530}), a nonconvex case is not well understood.

In this paper we investigate the structure of the signed fundamental solution of a nonconvex conservation law with a flux that satisfies (\ref{H1}). In Section \ref{sect.lemma} a structural lemma, Lemma \ref{lem.structure}, is introduced, which gives basic relations between convex-concave envelopes and the structure of   fundamental solution. This lemma is a summery of such relations in \cite{HaKim}. In Section \ref{sect.shocks}, the types of shocks are classified. There are four kinds of shocks, which are genuine shock, left contact, right contact and double contact. Each of these four kinds of shocks can be an increasing or decreasing one. Hence there are eight possibilities. For the convex flux case, there exists only one, which is a decreasing genuine shock. The dynamics of these shocks are introduced in Section \ref{sect.dynamics}. Branching is a phenomenon that a single shock is divided into two smaller shocks and merging is a one that two shocks are combined into a single shock of a smaller size. If the flux changes its sign as in (\ref{FluxForTrans}), the genuine shock can become a left or right contact. We call this phenomenon a transforming. In Section \ref{sect.rarefaction}, the types of rarefaction waves are discussed. For a convex flux case, a fundamental solution may have only a centered wave fan centered at the origin $(x,t)=(0,0)$. However, for a nonconvex flux case, the centered rarefaction wave fan can be placed at any place. Furthermore, there exists another kind of rarefaction wave which is called a contact rarefaction.

Finally, a complete scenario of an evolution of a fundamental solution is given in Section \ref{sect.full}. The flux and convex-concave envelopes of each stage are given in Figure \ref{fig6}. The whole evolution consists of eight stages which are divided by a merging, branching or transforming. The whole characteristic map is given in Figure \ref{fig7}. This characteristic map should be understood as an illustration. This dynamics of the fundamental solution can be observed by numerical computations. In fact one may find a numerical simulation result in Figure \ref{fig8}. In this computation the flux in Figure \ref{fig6} has been used.

\section{Structural lemma for fundamental solutions}\label{sect.lemma}

Define the speed of a shock that connects $u_1$ and $u_2$ by
$$
\sigma(u_1,u_2):={f(u_1)-f(u_2)\over u_1-u_2}.
$$
Suppose that an entropy solution $u(x,t)$ has a discontinuity along a curve $x=s(t)$. The discontinuity curve of an entropy solution or the discontinuity itself is called a \emph{shock} or shock curve. Then, the curve satisfies the Rankine-Hugoniot jump condition,
\begin{equation}\label{RH}
s^\prime (t) =\sigma(u_-(t),u_+(t)\,),\quad u_\pm(t)=\lim_{y\to s(t)\pm}u(y,t).
\end{equation}
A characteristic line $x=\xi(t)$ that emanates from a point $(x_0,t_0)$ satisfies
\begin{equation}\label{characteristic}
\xi^\prime(t) =f'(u(\xi(t),t)),\quad \xi(t_0)=x_0,\quad t\in I.
\end{equation}
If $I\subset[t_0,\infty)$, then the characteristic line $\xi(t)$ is called a forward one and, if $I\subset[0,t_0]$, then it is called a backward one.

The uniqueness of a signed fundamental solution has been shown by Liu and Pierre \cite[Theorem 1.1]{MR735207} for a Lipschitz continuous flux $\varphi$ such that $\varphi([0,\infty))\subset [0,\infty)$ and $\varphi(0)=0$. In the followings we first extend this uniqueness theorem under (\ref{H1}).

\begin{theorem}[Uniqueness of a signed fundamental solution]
\label{thm1} Suppose that the flux $f\in C^1(\bfR)$ satisfies (\ref{H1}). Then, there exists a unique fundamental solution $\rhom(x,t)$ that satisfies (\ref{eqn}) and (\ref{IC-rho}).
\end{theorem}
\begin{proof}
It is enough to show the uniqueness of a nonnegative solution for $m>0$. The assumption (\ref{H1}) and (\ref{nomalization}) imply that there exists $b>0$ such that $f(u)\ge -bu$ for all $u\ge0$.
Let $\varphi(u)=f(u)+bu$. Then $\varphi(0)=0$ and $\varphi(u)\ge0$
for all $u\ge0$. Therefore, there exists a unique nonnegative solution to
\begin{equation}\label{eqn2}
u_t+\varphi(u)_x=0,\qquad\lim_{t\to0} u(x,t)=m\delta(x)
\end{equation}
(see \cite[Theorem 1.1]{MR735207}). Let $u$ be the nonnegative solution and $v$ be its translation given by $v(x,t)=u(x-bt,t)\ge0$. Then,
$$
v_t+\varphi(v)_x=u_t-bu_x+f(u)_x+bu_x=0.
$$
One can easily check that $v$ also satisfies the entropy and the Rankine-Hugoniot jump conditions if $u$ does. Therefore, $v$ is an entropy solution to (\ref{eqn}). Let $\tilde v$ be another nonnegative solution and $\tilde u$ be given similarly by $\tilde v(x,t)=\tilde u(x-bt,t)$. Then, since a nonnegative solution to (\ref{eqn2}) is unique, $u(x,t)=\tilde u(x,t)$ and hence $v(x,t)=\tilde v(x,t)$. $\hfill\qed$\end{proof}

Now we show a scaling argument using this uniqueness theorem.

\begin{lemma} Let $\rhom$ be the unique fundamental solution of mass $m>0$. Then,
\begin{equation}\label{Invariance}
\rho_1(x,t)=\rhom(mx,mt),\quad x\in\bfR,\  t>0.
\end{equation}
\end{lemma}
\begin{proof}
Let $u(x,t)=\rhom(mx,mt)$ with $m>0$. Then,
\begin{eqnarray*}u_t=m\partial_t \rhom(mx,mt),\quad u_x=m\partial_x \rhom(mx,mt),\\
u_t+f'(u)u_x=m\partial_t \rhom(mx,mt) +mf'(\rhom(mx,mt))
\partial_x \rhom(mx,mt)=0.\\
\end{eqnarray*}
Hence, $u(x,t)$ is a solution. Furthermore, since
$$\int_{\bf R} \phi(x)u(x,0)dx =\int_{\bf R} \phi(x)\rhom(mx,0)dx =\int_{\bf R}\phi(y/m)\delta(y)dy =\phi(0)
$$
for any test function $\phi(x)$, $u(x,t)$ is the fundamental solution with mass $m=1$, i.e., $u=\rho_1$. Therefore, the uniqueness of a signed solution gives the relation (\ref{Invariance}).
$\hfill\qed$
\end{proof}
This lemma indicates that it is enough to study the structure of $\rho_1(x,t)$ only. We will denote the fundamental solution by $\rho(x,t)$ when we don't need to specify the size of mass $m>0$.

In the construction of a signed fundamental solution, the maximum value $\bar\rho$ and the maximum point $\zeta(t)$, i.e.,
$$
\bar\rho(t):=\sup_x \rho(x,t),\quad \max\{\rho(\zeta(t)-,t),\rho(\zeta(t)+,t)\}=\bar\rho(t),
$$
are used as parameters and then decided implicitly. The convex and concave envelopes are respectively defined by
\begin{equation}\label{envelope}
 h(u;\bar\rho) :=\sup_{\eta\in A(0,\bar\rho)}\eta(u),\quad
 k(u;\bar\rho) :=\inf_{\eta\in B(0,\bar\rho)}\eta(u),
\end{equation}
where
\begin{eqnarray}\label{setAB}
A(0,\bar\rho):=\{\eta:\eta''(u)\ge0,~\eta(u)\le
f(u)\mbox{~~for~~}0<u<\bar\rho\},\\
B(0,\bar\rho):=\{\eta:\eta''(u)\le0,~\eta(u)\ge
f(u)\mbox{~~for~~}0<u<\bar\rho\}.
\end{eqnarray}
One can easily check that, for any fixed $\bar\rho>0$, $h(u;\bar\rho)$ and $k(u;\bar\rho)$ are convex and concave functions on the interval $(0,\bar\rho)$, respectively. Since we consider a flux with a finite number of inflection points, the domain $(0,\bar\rho)$ can be divided into a finite number of subintervals so that envelopes are identical to the flux or a line on each subinterval.

Finally, the fundamental solution is given by the inverse relation of
\begin{equation}\label{NewFsol}
\left\{\begin{array}{ll}
\bar h(\rho(x,t),t)&=x\mbox{~~for~} x<\zeta(t),\\
\bar k(\rho(x,t),t)&=x\mbox{~~for~} x>\zeta(t),\\
\end{array}
\right.
\end{equation}
where $\bar h$ and $\bar k$ satisfy
\begin{equation}\label{barhprime}
\partial_t\bar h(u,t)=\partial_u h(u;\bar \rho(t))),\quad
\partial_t\bar k(u,t)=\partial_u k(u;\bar \rho(t)))
\end{equation}
(see \cite[Section 4]{HaKim} for details). The structure of the fundamental solution is analyzed in the rest of the paper using the dynamics and the relations of the envelopes. The following lemma is a summary of the basic relations and dynamics. We will use this structural lemma in analyzing the components and dynamics of fundamental solutions.

\begin{lemma} [structural lemma for fundamental solutions] \label{lem.structure} Let $0=a_0<a_1<\cdots<a_{i_0}=\bar \rho(t)$ be the minimal partition of $[0,\bar \rho(t)]$ such that the convex envelope $h(u;\bar \rho(t))$ is either linear or identical to $f(u)$ on each subinterval $(a_{i},a_{i+1})$, $0\le i<i_0$. Similarly, let $0=b_0<b_1<\cdots<b_{j_0}=\bar \rho(t)$ be the minimal partition related to the concave envelope $k(u;\bar \rho(t))$. Let $\zeta_0(t)$ is the maximum point in the sense that $\bar \rho(t)=\max(\rho(\zeta_0(t)+,t),\rho(\zeta_0(t)-,t))$ and $\spt(\rho(\cdot,t))=[\zeta_-(t),\zeta_+(t)]$.

$(i)$ The linear parts of the envelopes are tangent to the flux, i.e.,
\begin{eqnarray*}\label{tangent}
h'(a_i;\bar \rho(t))=f'(a_i),\qquad i=1,\cdots,i_0-1,\\
k'(b_j;\bar \rho(t))=f'(b_j),\qquad j=1,\cdots,j_0-1.
\end{eqnarray*}

$(ii)$ The maximum $\bar \rho(t)$ is strictly decreasing as $t\to\infty$.

$(iii)$ The solution $\rho(x,t)$ increases in $x$ on the interval $(\zeta_-(t),\zeta_0(t))$. If $h(u;\bar \rho(t))$ is linear on $(a_i,a_{i+1})$, $\rho(x,t)$ has an increasing discontinuity that connects $u_-=a_i$ and $u_+=a_{i+1}$. If $f(u)=h(u;\bar \rho(t))$ on $(a_i,a_{i+1})$, $\rho(x,t)$ has a rarefaction profile that continuously increases from  $u=a_i$ to $u=a_{i+1}$.

$(iv)$ The solution $\rho(x,t)$ decreases in $x$ on the interval $(\zeta_0(t),\zeta_+(t))$. If $k(u;\bar \rho(t))$ is linear on $(b_j,b_{j+1})$, $\rho(x,t)$ has a decreasing discontinuity that connects $u_-=b_{j+1}$ and $u_+=b_{j}$. If $f(u)=k(u;\bar \rho(t))$ on $(b_j,b_{j+1})$, $\rho(x,t)$ has a rarefaction profile that continuously decreases from  $u=b_{j+1}$ to $u=b_{j}$.
\end{lemma}

Lemma \ref{lem.structure} gives the dynamics of the fundamental solution in terms of convex-concave envelopes. First, the number of discontinuities and their left and right hand limits are given by the convex-concave envelopes if the maximum $\bar \rho(t)$ at a specific time $t>0$ is known. However, we do not know the location of a discontinuity. The discontinuities are connected by rarefaction waves. However, if the flux is not convex, the structure of rarefaction waves are quite complicated and are not functions of $x/t$ anymore (see Section \ref{sect.rarefaction}).
\begin{figure}[htb]
\begin{minipage}[t]{5.5cm}
\centering \psfrag  {u(t)}{$\bar \rho(t)$} \psfrag  {b1}{$b_1$}
\psfrag  {b2}{$b_2$} \psfrag  {b3}{$b_3$} \psfrag  {b4}{$b_4$}
\psfrag  {a1}{$a_1$} \psfrag  {a2}{$a_2$} \psfrag  {a3}{$a_3$}
\psfrag  {f(u)}{$f(u)$} \psfrag  {0}{\scriptsize $0$}
\includegraphics[width=5.5cm]{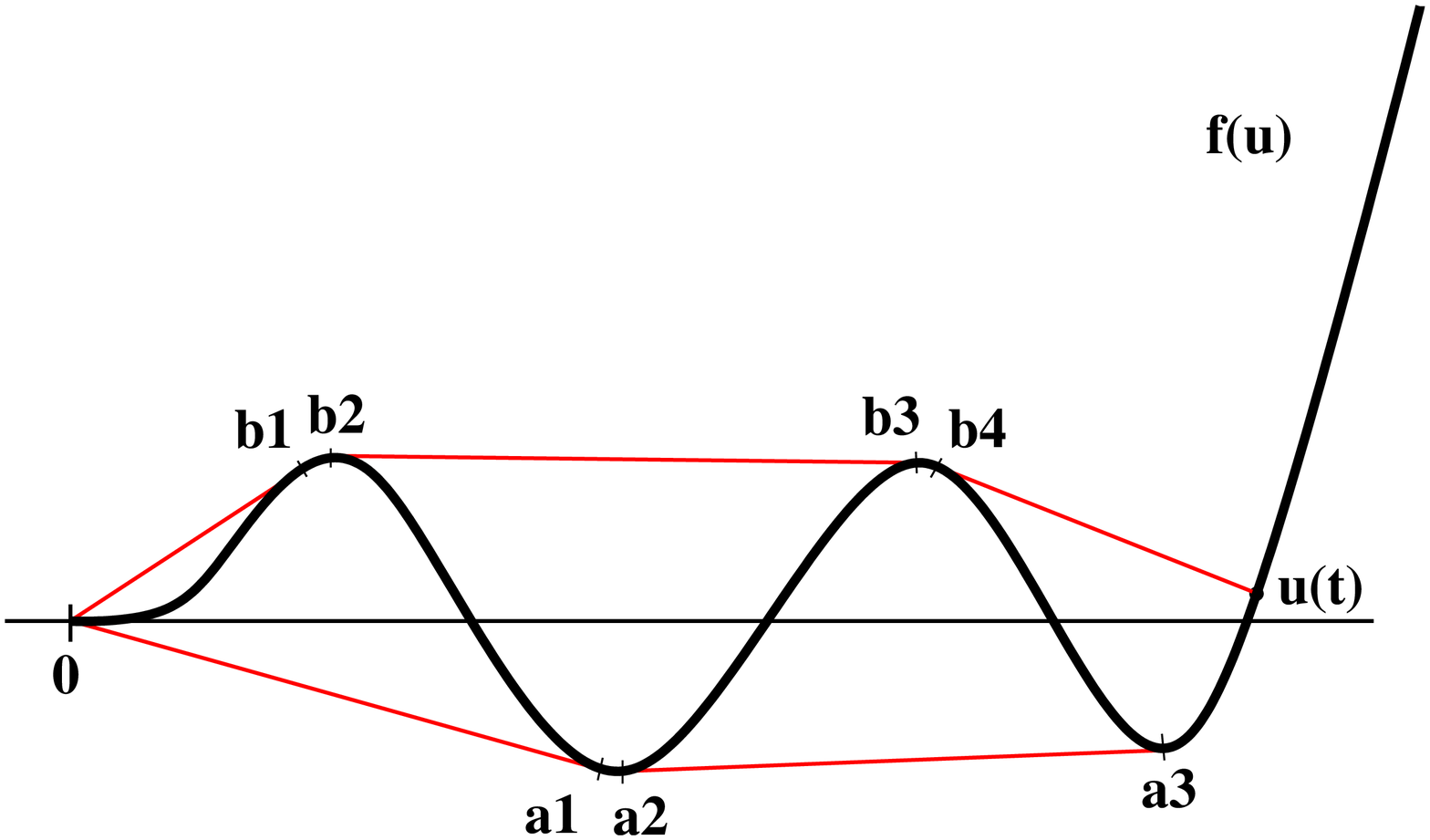}
(a) continuous change of concave envelopes
\end{minipage}
\hfill
\begin{minipage}[t]{5.5cm}
\centering \psfrag  {b1}{$b_1$} \psfrag  {b2}{$b_2$}
\psfrag  {b3}{$b_3$} \psfrag  {b4}{$b_4$}
\psfrag  {b5}{\hskip -8mm $\bar\rho(t-)=b_1$} \psfrag  {a1}{$a_1$}
\psfrag  {a2}{$a_2$}
\psfrag  {a3}{$a_3=\bar\rho(t+)$} \psfrag  {f(u)}{$f(u)$}
\psfrag  {0}{\scriptsize $0$}\psfrag {u(t)}{$\bar \rho(t)$}
\includegraphics[width=5.5cm]{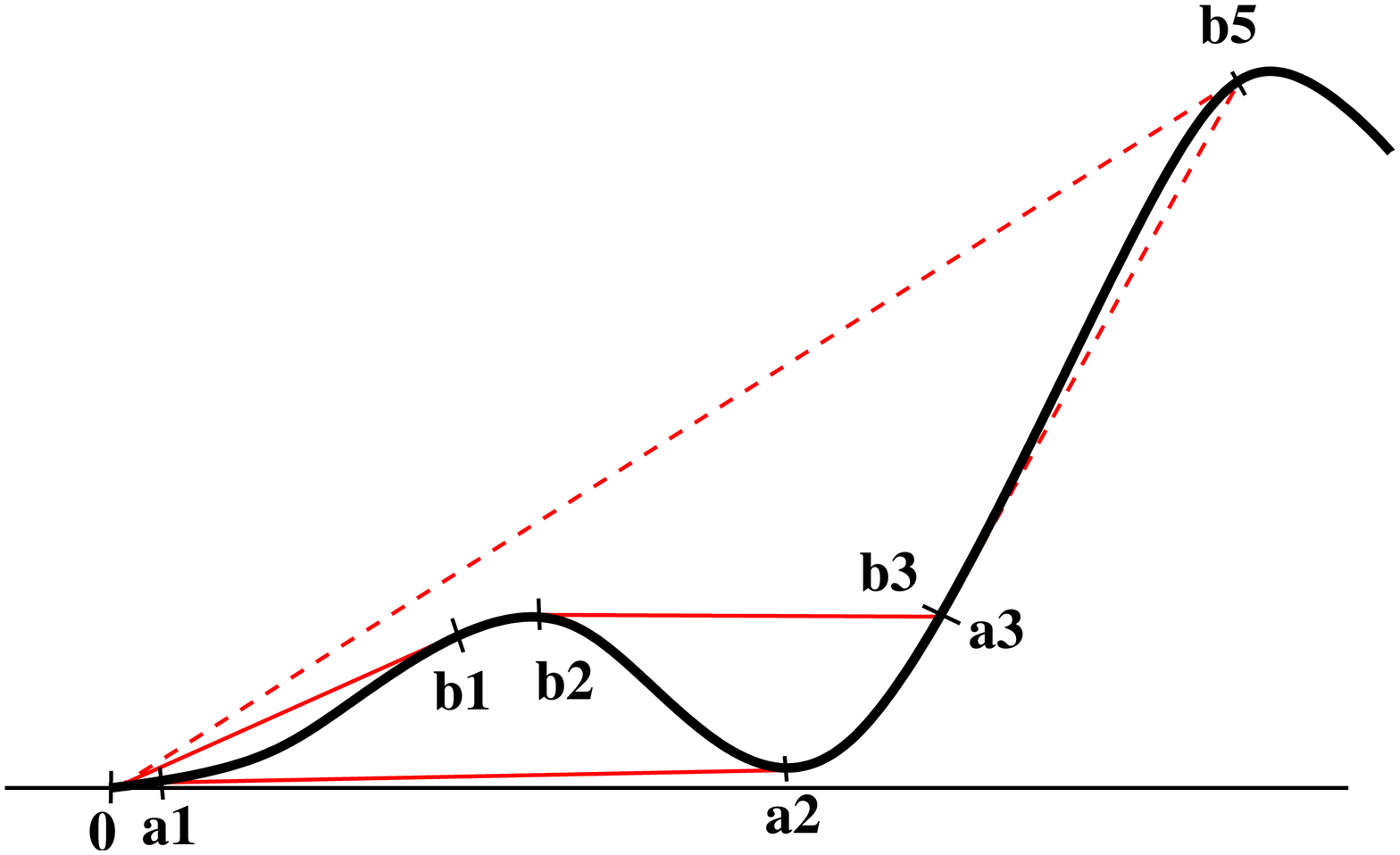}
(b) discontinuous change of concave envelopes
\end{minipage}
\caption{Examples of envelopes and minimal partitions. consist of $0,\
\bar \rho(t)$ and coordinates of the horizontal axis (or $u$-axis) of
the tangent points. If both of the envelopes meet with linear
parts as in (b), then the maximum jumps to the nearest interior
partition point $a_3$ in the figure.}\label{fig1}
\end{figure}

In Figure \ref{fig1}(a), convex-concave envelopes on a given domain $[0,\bar \rho(t)]$ are illustrated. Note that, the graph of the flux $f(u)$ is tangent to the $u$-axis at the origin since (\ref{nomalization}) is assumed. The minimal partition values $a_i$'s for the convex envelope are marked at the corresponding tangent points. Two linear parts of the convex envelope indicate that the fundamental solution have two increasing discontinuities. One of them jumps from $0$ to $a_1$ and the other from $a_2$ to $a_3$. These discontinuities satisfy the entropy condition. Similarly the concave envelope and the corresponding minimal partition $b_j$'s provide the decreasing shocks.

As the maximum $\bar \rho(t)$ decreases, the corresponding envelopes evolve continuously. However, if the end point $(\bar \rho(t),f(\bar \rho(t)))$ reaches to a tangent point, then the envelopes change discontinuously and an example is given in Figure \ref{fig1}(b). The envelopes in dashed lines are the case when $\bar \rho(t-)=b_1$. This implies that the maximum value $\bar \rho(t-)$ is connected to the value $a_3$ by an increasing shock and also connected to $0$ by a decreasing shock. In other words the solution has an isolated singularity. The case $a_3<\bar \rho<b_1$ is not admissible by the same reason. Therefore $\bar \rho$ jumps from $b_1$ to $a_3$ and the envelopes jump from the dashed ones to the solid ones in the figure. In particular the minimal partition for the concave envelope has new members and should be re-indexed as in the figure.

\section{Classification of shocks}\label{sect.shocks}

The dynamics of shocks is the key in understanding the structure of a fundamental solution of a conservation law. In this section we classify the types of shocks. Let $x=s(t)$ be a shock curve of $\rho(x,t)$ and $u_0^\pm=\lim_{\eps\downarrow0} u(s(t_0)\pm\eps,t_0)$ be one-sided limits. Let $x=\xi_+(t)$ be the maximal characteristic curve and $x=\xi_-(t)$ be the minimal one, where both of them emanate from the given point $(s(t_0),t_0)$. Then, the backward or the forward characteristic curves satisfy
\begin{equation}\label{speed}
\xi_+'(t_0)=f'(u_0^+),\quad \xi_-'(t_0)=f'(u_0^-),
\end{equation}
where the derivatives of characteristic curve are understood as one sided ones depending on its domain. We are interested in the characteristic curve that satisfies (\ref{speed}) since they are the ones that carry the information.

One may obtain the following well known relations from the Oleinik entropy condition,
\begin{equation}\label{inequalities}
f'(u^-_0)\ge s'(t_0)\ge f'(u^+_0).
\end{equation}
In the followings we classify the shocks into four types.

\subsection{Genuine shocks}

If both inequalities in (\ref{inequalities}) are strict, i.e.,
\begin{equation}\label{G}
f'(u^-_0)>s'(t_0)>f'(u^+_0),
\end{equation}
then the shock curve $x=s(t)$ is called a {\it genuine shock} and denoted by the letter `G' in figures. If a faster characteristic line $x=\xi_-(t)$ collides to the slower shock curve $x=s(t)$ from the left, it should come from the past (or as $t\to t_0-$). Similarly, if the slower characteristic line $x=\xi_+(t)$ collides to the shock curve from the right, then it  should also come from the past. Therefore, for a genuine shock case, characteristics satisfying (\ref{speed}) are backward ones and
\begin{equation}\label{genuineshock}
\xi_-'(t)>s'(t)>\xi_+'(t),\quad t_0-\eps<t<t_0
\end{equation}
for some $\eps>0$. If one may take the domain as $0<t<t_0$, then the characteristic curve is called global, which is always the case with a convex flux. However, if the flux is nonconvex, the characteristic curves are not necessarily global.

\begin{property}
A signed fundamental solution has at most one genuine shock, and it has one if and only if the convex or the concave envelope is a non-horizontal line. Furthermore, the genuine shock always connects the maximum and the zero.
\end{property}
\begin{proof}
Suppose that a shock connects a value of an intermediate partition point, say $a_i$ with $0\ne i\ne i_0$. Since the shock speed $s'(t)$ is given by the relation in (\ref{RH}), Lemma \ref{lem.structure}($i$) gives that
$$
s'(t)=h'(a_i;\bar \rho(t))=f'(a_i)=\xi'(t),\quad t_0-\eps<t<t_0.
$$
Therefore, at least one of the inequalities in (\ref{inequalities}) is an equality and hence the shock is not a genuine one. If a genuine shock connects the zero and the maximum, the corresponding envelope should be a line. Furthermore, since both envelopes can not be lines at the same time, a fundamental solution has at most one genuine shock at any given time.
$\hfill\qed$
\end{proof}

If the convex or concave envelope is a horizontal line, then due to the normalization (\ref{nomalization}), one of the inequalities in (\ref{inequalities}) is an equality. This is a transition stage of a genuine shock into a contact discontinuity which will be discussed in Section \ref{sect.dynamics}. Now suppose that a concave envelope is a non-horizontal line. Of course, the discontinuity connects the zero value and the maximum. We may easily see that if the line is not tangent to the graph of the flux, then it gives a genuine shock. Suppose that the line is tangent to the flux at the maximum as in Figure \ref{fig1}(b). Then the flux is locally concave near the maximum value $u=\bar \rho(t)$ and hence the convex envelope is also linear at the point $(\bar \rho(t),f(\bar \rho(t)))$. This implies that the maximum is an isolated singularity which is not admissible. Therefore, such envelopes do not exist. The same arguments are applied to the convex envelope.

If the flux is convex, its concave envelope is simply a non-horizontal line and gives a decreasing genuine shock all the time. Furthermore, there are no other types of shocks for the convex flux. Figure \ref{fig2}(a) is an illustration of a genuine shock.

\begin{figure}[htb]
\begin{minipage}[t]{2.6cm}
\centering \psfrag  {G.S.}{G}
\includegraphics[height=3.5cm]{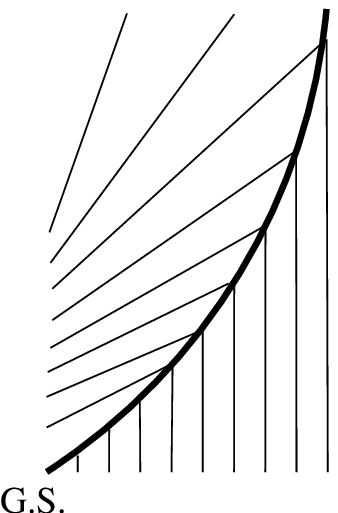}

(a) genuine shock
\end{minipage}
\hfill
\begin{minipage}[t]{2.8cm}
\centering \psfrag  {II}{D}
\includegraphics[height=3.5cm]{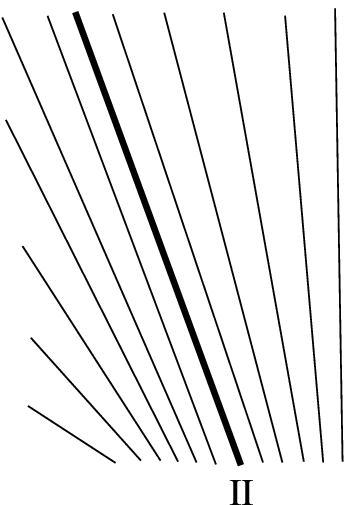}

(b) double contact
\end{minipage}
\hfill
\begin{minipage}[t]{2.6cm}
\centering \psfrag  {R}{R}
\includegraphics[height=3.5cm]{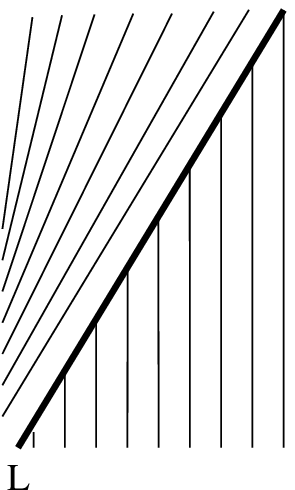}

(c) left contact
\end{minipage}
\hfill
\begin{minipage}[t]{2.6cm}
\centering \psfrag  {L}{L}
\includegraphics[height=3.5cm]{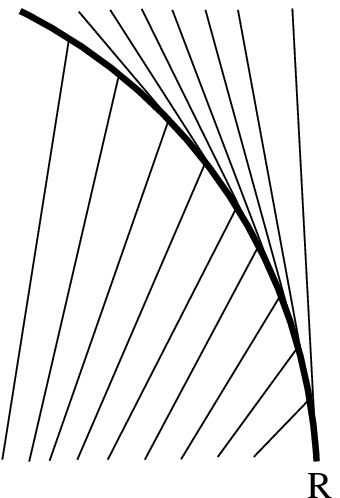}

(d) right contact
\end{minipage}
\caption{[The horizontal axis is for space $x$ and the vertical axis is for time $t$.] Shocks of a nonconvex scalar conservation law are classified into four types.} \label{fig2}
\end{figure}

\subsection{Contact shocks}

A shock is called a \emph{contact} if it is not a genuine shock. If both inequalities in (\ref{inequalities}) are equalities, i.e.,
\begin{equation}\label{D}
f'(u^-_0)=s'(t_0)=f'(u^+_0),
\end{equation}
then the shock is called a \emph{double sided contact} or simply \emph{double contact} and denoted by `D' in figures. One can easily find such a discontinuity from convex-concave envelopes. Let $a_i,\ i=0,\cdots,i_0,$ be the minimal partition in Lemma \ref{lem.structure} related to the convex envelope $h$. One can easily see that if $h$ is linear in an interior subinterval $(a_i,a_{i+1})$ (i.e., $a_i\ne0$ and $a_{i+1}\ne\bar \rho(t)$), then Lemma
\ref{lem.structure}($i$) implies that (\ref{D}) is satisfied. In this case the shock is placed between two rarefaction waves. Since the shock speed is constant until the tangent point $a_{i+1}$ stays in the partition, the double sided contact is a line parallel to adjacent characteristic lines (see Figure \ref{fig2}(b)). One may similarly consider double sided contacts related to the concave envelope which is omitted.

If one of the inequalities in (\ref{inequalities}) is an equality and the other one is a strict inequality, then we call it a \emph{single sided contact}. If
\begin{equation}\label{L}
f'(u^-_0)=s'(t_0)>f'(u^+_0),
\end{equation}
then this single sided contact is called a \emph{left sided contact} or simply \emph{left contact} and denoted by `L' in figures. This means that the characteristic lines on the left hand side of the shock curve have the same speed as the one of the shock. Hence the characteristic lines are tangent to the shock curve from the left hand side  (see Figures \ref{fig2}(c)).

Similarly, if
\begin{equation}\label{R}
f'(u^-_0)>s'(t_0)=f'(u^+_0),
\end{equation}
then this single sided contact is called a \emph{right sided} or \emph{right contact} and denoted by `R' in figures.

\begin{property}
A signed fundamental solution has two or three single sided contacts counting a genuine shock as two. If there is no genuine shock at a moment, then there exist right  and left contacts for each at least.
\end{property}
\begin{proof}
A single sided contact should connect the maximum $\bar \rho(t)$ or the zero value to an interior partition value. If a discontinuity connects two interior partition points, then it is a double sided contact as discussed before. Since the maximum can not be connected by two shocks (see the comments following Figure \ref{fig1}(b)), the total number of contact shocks is at most three.
$\hfill\qed$
\end{proof}

\begin{property} Double sided contacts are lines. Single sided contacts connected to zero are lines. There exist exactly one shock connected to the maximum $\bar\rho(t)$, which is the only one that moves with a nonconstant speed.
\end{property}
\begin{proof}
We have already observed that the double sided contacts are lines. Single sided contacts connected to the zero value are also lines by the same reason that the linear part of the corresponding envelope is not changed as long as the contact is a single sided one.  For example, consider a right contact case. Then, since $0$ is the minimum, the contact is an increasing shock and hence the convex envelope is considered.  Since the shock speed $s'(t)$ is given by the relation in (\ref{RH}), Lemma \ref{lem.structure}($i$) gives that
$$
s'(t)=h'(a_1;\bar \rho(t))=f'(a_1),
$$
where $a_1$ is the first interior partition point, which is constant as long as the right contact remains as it is. Hence, the right contact is a line. The left contact case is similar (see Figure \ref{fig2}(c)).

Now consider a left contact that connects the maximum $\bar\rho(t)$. Then, it is an increasing shock that connects the maximum to the last interior partition point $a_{i_0-1}$. Hence the shock speed is
$$
s'(t)=h'(a_{i_0-1};\bar \rho(t))=f'(a_{i_0-1}),
$$
Since the maximum of the fundamental solution $\bar\rho(t)$ strictly decreasing, the slope of the convex envelope at the maximum point is strictly increasing. Hence the left contact is not a line, but is curved to right since the propagation speed is increasing. One can show the similar behavior for the right contacts and the difference is that the right contact is curved to left (see Figure \ref{fig2}(d)). The genuine shock is similar as the contacts that connects to the maximum value $\bar\rho(t)$.
$\hfill\qed$\end{proof}

According to the previous properties, there exists only one shock that connects the maximum $\bar \rho(t)$. This shock is a genuine shock or a single sided contact that moves along a curve. All the other shocks propagate with a constant speed. Hence, all the dynamics of shock waves are produced along the shock that connects the maximum $\bar\rho(t)$ which is the only curved one of the signed  fundamental solution.

\begin{example}
Consider the linear part of the concave envelope in Figure \ref{fig1}(a) that connects the maximum and an interior partition value $b_4$. First, note that the right hand side limit of the shock is $b_4$ since the concave envelope gives a decreasing shock. The speed of the characteristic line carrying this value is $f'(b_4)$ which is identical to the shock speed and hence the corresponding discontinuity is always a right contact. One can easily see that the slope of the linear part decreases as $\bar \rho(t)$ decreases (i.e., as $t$ increases). Therefore the shock curve makes a turn to the left hand side as $t$ increases like in Figure \ref{fig2}(d). Furthermore, the interior tangent value $b_4$ increases, which indicates that the range covered by rarefaction wave is increasing. In other words new information is produced and propagates to the future. Therefore, the characteristic line $x=\xi_+(t)$ touching the shock from the right hand side has a domain $t\in(t_0,t_0+\eps)$ for some $\eps>0$. In Figure \ref{fig2}(d) this kind of right contact has been illustrated. Even if the previous discussions are in terms of the concave envelope, one may repeat them for the convex envelope and obtain the dual statements.
\end{example}

\section{Dynamics of shocks}\label{sect.dynamics}
The convex-concave envelopes have one end at the origin and the other end at the maximum point $(\bar \rho(t), f(\bar \rho(t)))$ with $\bar \rho(t)=\sup_x \rho(x,t)$. Since the maximum $\bar \rho(t)$ decreases in time $t$, the envelopes and the corresponding minimal partitions changes. In the followings we consider the dynamics of shock curves by tracking these changes.

\subsection{Branching}\label{sect.bran}
A shock curve may split into two smaller shocks divided by a rarefaction wave and we call this phenomenon a \emph{branching}. Consider a shock curve that connects the maximum value $\bar \rho(t)$. Then it should be the genuine shock or a single sided contact. If the corresponding linear part of the envelope touches a hump of the graph of the flux $f(u)$ on its way (see Figures \ref{fig6}(b) and \ref{fig6}(c)\,), it will split into two linear parts with a convex or a concave part in between. Since both of these two linear parts belong to the convex envelope or concave envelope, both shocks are increasing ones or decreasing ones. In other words an increasing shock splits into two smaller increasing shocks and a decreasing one into two smaller decreasing ones.

One can easily see that, at the moment the branching process starts, the linear parts of the envelope corresponding to the incoming and outgoing shocks are all the same line. Hence, the slopes of shock curves at the branching point in the $xt$-plane are identical and hence they form smooth curves of branching as in Figure \ref{fig3}.

Since the incoming shock is connected to the maximum $\bar\rho(t)$, one of the outgoing shocks connects the maximum, which should be a single sided contact. If the incoming shock is a single sided contact, then the other outgoing shock is a double sided contact as in Figure \ref{fig3}(a). Therefore, we may conclude that if the incoming shock is of single sided, it splits into one single sided and one double sided contacts. Similarly, if the incoming shock is a
genuine shock, then it splits into two single sided contacts (see Figure \ref{fig3}(b)). Note that type D doesn't split.

\begin{figure}[htb]
\begin{minipage}[t]{5cm}
\centering \psfrag  {R}{R} \psfrag  {II}{D}
\includegraphics[height=3.5cm]{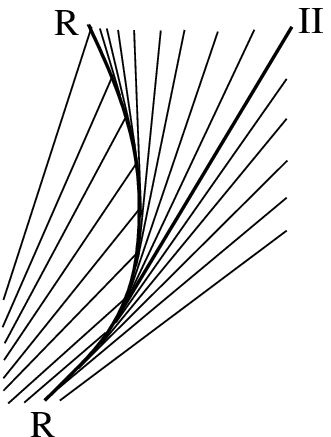}

(a)Branching (R$\to$R+D)
\end{minipage}
\hfill
\begin{minipage}[t]{5cm}
\centering \psfrag  {G.S.}{G} \psfrag  {R}{R}
\psfrag  {L}{L}
\includegraphics[height=3.5cm]{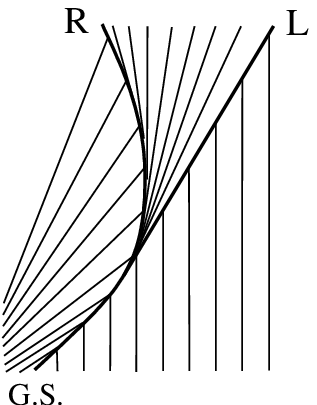}

(b)Branching (G$\to$R+L)
\end{minipage}
\caption{[The horizontal axis is for space $x$ and the vertical axis is for time $t$.]  There are three kinds of branching process. The third one is `L$\to$L+D'.}\label{fig3}
\end{figure}

\begin{property}\label{P.branching}
There are three situations of branching classified by the incoming and outgoing shocks. They are:
$(i)$ R$\to$R+D,
$(ii)$ L$\to$L+D,
$(iii)$ G$\to$R+L.
\end{property}

\subsection{Merging}\label{sect.merg}
Two shocks can be combined and then form a single shock. We call this phenomenon a \emph{merging}. In the process these two shocks have different monotonicity. One may see this phenomenon from the change of envelopes. As the maximum value $\bar \rho(t)$ decreases, two linear parts of convex and concave envelopes may meet at a point, say $(\bar \rho(t_0), f(\bar \rho(t_0)))$ (see Figures \ref{fig1}(b) or
\ref{fig6}(d)). However, in this case it gives a
removable jump (see \cite[Lemma 3$(iii)$]{HaKim}) and hence the maximum of the fundamental solution has a decreasing jump from $\bar \rho(t_0+)$ to $\bar \rho(t_0-)$. In this case $\bar \rho(t_0-)$ is the largest interior partition point (e.g., the point $a_3$ in Figure \ref{fig1}(b)). One can easily see that the slope of the linear parts of envelopes related these two incoming shocks and one outgoing shock are all distinct and hence the shock curves are not smooth in general. The phenomenon related to this sudden change of envelopes will be discussed later for the aspect of a rarefaction wave.

\begin{figure}[htb]
\begin{minipage}[t]{3.5cm}
\centering \psfrag  {G.S.}{G} \psfrag  {L}{L}
\includegraphics[height=3.5cm]{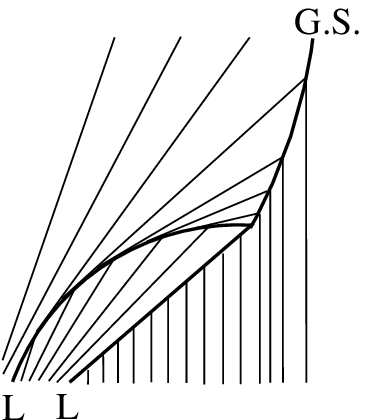}

(a)Merging (L+L$\to$G)
\end{minipage}
\hfill
\begin{minipage}[t]{3.5cm}
\centering \psfrag  {R}{R}
\psfrag  {L}{L} \psfrag  {II}{D}
\includegraphics[height=3.5cm]{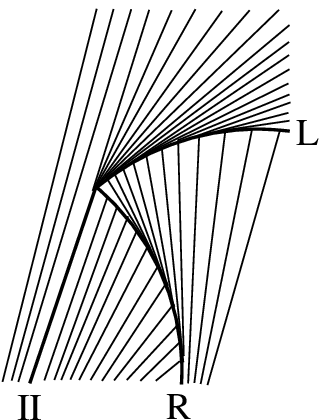}

(b)Merging (R+D$\to$L)
\end{minipage}
\hfill
\begin{minipage}[t]{3.5cm}
\centering \psfrag  {L}{L} \psfrag  {R}{R}
\includegraphics[height=3.5cm]{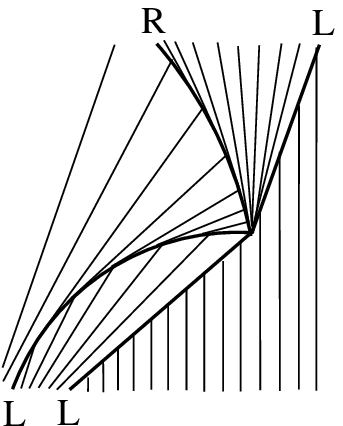}
(c)Merging+Branching
\end{minipage}
\caption{[The horizontal axis is for space $x$ and the vertical axis is for time $t$.] There are two kinds of merging. The last figure shows an example when merging and branching occur simultaneously.}\label{fig4}
\end{figure}

In Figure \ref{fig4}(a) two left contacts merge into a genuine shock. If a right contact is merged with a double sided contact, then a left contact is produced as in Figure \ref{fig4}(b). Note that merging is not an opposite process of branching. In a merging process, two shocks of different monotonicity produces a single smaller shock. Remember that, in a branching process, two shocks of the same monotonicity are produced. Another difference is that the shock
curves are not smooth after a merging process.

\begin{property}\label{P.merging}
There are four situations of merging classified by the incoming and outgoing shocks. They are:
$(i)$   R+R$\to$G,
$(ii)$  L+L$\to$G,
$(iii)$ R+D$\to$L, and
$(iv)$ L+D$\to$R.
\end{property}

\subsection{Merging $+$ Branching}

Merging and branching are basic phenomena in the dynamics of discontinuities. These two phenomena may appear at the same time. Since the envelopes may change discontinuously after a merging process, a branching may follow immediately after it. Consider the example in Figure \ref{fig1}(b), where two incoming shocks meet at a point. Then, the merging process in the figure produces one left contact and one right contact, where the corresponding figure is in Figure \ref{fig4}(c). Notice that the outgoing shocks are of the same monotonicity. In this example they are given by the concave envelope and hence they are decreasing ones. It is also possible that there are more than two
out going shocks. For example, if there are many wiggles in the inside hump of Figure \ref{fig1}(b), then there can be many outgoing shocks with same monotonicity with each other. However, those cases all extra smaller contacts are double sided contacts. From Properties \ref{P.branching} and \ref{P.merging}, we obtain the following property.

\begin{property}\label{P.branching+merging}
There are four possible situations of `branching after merging' classified by the incoming and outgoing shocks. They are:
$(i)$  R+R$\to$R+L,
$(ii)$  L+L$\to$R+L,
$(iii)$ R+D$\to$L+D,
$(iv)$ L+D$\to$R+D.
\end{property}

\subsection{Transforming}\label{sect.tranforming}

A shock may change its type without branching or merging and we call this phenomenon a \emph{transforming}. This phenomenon may appear only if \begin{equation}\label{FluxForTrans}
f\big([0,\infty)\big)\not\subseteq[0,\infty).
\end{equation}

The only possible case is that a genuine shock is transformed to a single sided contact. It always happens when a genuine shock changes its direction from the negative one to the positive one or in the other way. For example consider a genuine shock that moves at a negative speed as in Figure \ref{fig5}(a). If it stops and then moves to the positive direction, then it is not a genuine shock any more. It becomes a left contact as one can see from the figure. Similarly, if a genuine shock changes its direction from the positive one to the negative one,
then it becomes a right contact.

In Figure \ref{fig4}(a) two single sided contacts are merged into a genuine shock. If transforming appears simultaneously, then one may see the phenomenon that two contacts of single sided are merged into a single contact of single sided (see Figure \ref{fig5}(b)). Notice that the phenomenon in Figure \ref{fig5}(b) can be considered as a special case of merging process. However, instead of placing it in the section for merging, we have it in this transforming section since it is based on the transforming process. Furthermore, the transforming phenomena can be found under the extra hypothesis (\ref{FluxForTrans}). In fact, if the flux is nonnegative, then the speed of the genuine shock is positive and hence transforming phenomenon does not appear. This phenomenon can be found in the transition from Figure \ref{fig6}(f) to  \ref{fig6}(g).

\begin{figure}[htb]
\begin{minipage}[t]{4cm}
\centering \psfrag  {G.S.}{G} \psfrag  {L}{L}
\includegraphics[height=4cm]{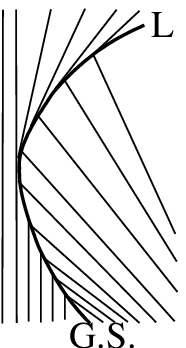}

(a) Transforming: G$\to$L
\end{minipage}
\hfill
\begin{minipage}[t]{6cm}
\centering \psfrag  {L}{L}
\psfrag  {R}{R}
\includegraphics[height=4cm]{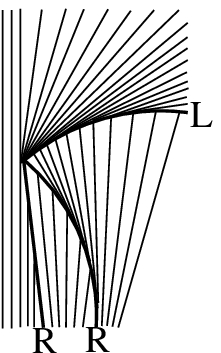}

(b) Merging+Transforming: R+R$\to$L
\end{minipage}
\caption{[The horizontal axis is for space $x$ and the vertical axis is for time $t$.] Transforming can be observed under an extra condition (\ref{FluxForTrans}).}
\label{fig5}
\end{figure}

\begin{property}\label{P.transforming+merging}
There are four possibilities under (\ref{FluxForTrans}) related to transforming. They are:
$(i)$  G$\to$R,
$(ii)$  G$\to$L,
$(iii)$ R+R$\to$L,
$(iv)$ L+L$\to$R.
\end{property}

\section{Structure of rarefaction waves}\label{sect.rarefaction}

There is no contact shock for a convex flux case and hence all the characteristics of a fundamental solution carrying the information of a non-zero value are emanated from the origin. Therefore, if a fundamental solution $\rho$ has a rarefaction profile at a point $(x,t)$, the speed of the characteristic line that passes through the point is $x/t$ and hence it should be satisfied
that $f'(\rho(x,t))=x/t$. Furthermore, since $f'$ is invertible if the flux is strictly convex, the rarefaction wave should be given by the following relation
\begin{equation}\label{convex-rarefaction}
\rho(x,t)=(f')^{-1}(x/t),\qquad a(t)\le x\le b(t),
\end{equation}
where $[a(t),b(t)]$ is the support of the fundamental solution $\rho(\cdot,t)$. However, if the flux is not convex, then there may exist contact shocks and hence there are various possibilities for the starting point of the characteristic line. Furthermore, since $f'$ is not invertible in the whole domain, one should clarify the correct profile that gives the rarefaction wave. In the followings we classify the rarefaction waves.

\subsection{Centered rarefaction wave fans}

There are two kinds of centered rarefaction waves. The first one is the one produced from the initial profile of the Dirac-measure and hence centered at the origin. Let $h(u;\infty)$ be the convex envelope of the flux and $0=a_0<a_1<\cdots<a_{i_0}<\infty$ be the minimal partition. Then the rarefaction wave is given as
\begin{equation}\label{initial-rarefaction}
\rho(x,t)=g_\infty(x/t),\qquad a(t)\le x\le b(t),\ 0<t<\eps,
\end{equation}
where the similarity profile $g_\infty(x)$ is a piecewise continuous function that satisfies $h'(g_\infty(x);\infty)=x$. Notice that the function $g_\infty(x)$ may have a discontinuity and hence $\rho(x,t)$ given by (\ref{initial-rarefaction}) may have a discontinuity which is actually contact. Therefore, the rarefaction wave in (\ref{initial-rarefaction}) should be understood as a sequence of rarefaction waves divided by contacts. In Figure \ref{fig7} an example of initial centered rarefaction wave can be found for $t>0$ small. In the figure one may find the wave fan bounded by a genuine shock and a right contact. One may also find the inside profile is divided by a double sided contact.

A centered rarefaction wave may also appear after a merging process. If a shock collides to another one, then the envelopes change discontinuously and a centered rarefaction wave may emerge. For example consider the merging process in Figure \ref{fig1}(b) and let $(x_0,t_0)$ be the merging point. Then the concave envelope jumps from the dashed one to the solid ones and a rarefaction part in the interval $(b_2,b_3)$ is added to the concave envelope, Figure \ref{fig1}(b), which generates a centered rarefaction wave given as
\begin{equation}\label{dynamic-rarefaction}
\rho(x,t)=g_1\Big({x-x_0\over t-t_0}\Big),\quad \xi(t)\le x\le
s(t),\quad t_0<t<t_0+\eps,
\end{equation}
where $g_1$ is the inverse function of the flux on the domain $(b_2,b_3)$ and the wave is bounded by a contact line of single sided and a characteristic
line, which are given by
$$
\xi(t)=x_0+f'(b_2)(t-t_0),\ s(t)=x_0+{ f(b_1)\over b_1
}(t-t_0),\ t_0<t<t+\eps.
$$
The wave fan which is between two outgoing contacts of single sided in Figure \ref{fig4}(c) and emanates from the branching point is a corresponding case.

One may also observe a centered rarefaction wave bounded by two characteristic lines. Consider the change of envelopes in Figure \ref{fig6}(d) after a merging. Then the interior partition point $a_2$ jumps from $a_2^-$ to $a_2^+$ and a rarefaction part in the interval $(a_2^-,a_2^+)$ is added to the convex envelope, which generates a centered rarefaction wave given by
\begin{equation}\label{dynamic-rarefaction2}
\rho(x,t)=g_2\Big({x-x_0\over t-t_0}\Big),\quad \xi_1(t)\le x\le
\xi_2(t), \quad t_0<t<t+\eps,
\end{equation}
where $g_2$ is the inverse function of the derivative of the convex envelope of the flux on the domain $(a_2^-,a_2^+)$ and the wave fan is bounded by two characteristic lines, which are given
by
$$
\xi_1(t)=x_0+f'(a_2^-)(t-t_0),\quad
\xi_2(t)=x_0+f'(a_2^+)(t-t_0),\quad t_0<t<t+\eps.
$$

\begin{remark}
The solution is continuous along the characteristic $\xi_1$, but not differentiable. This kind of regularity has been mentioned in Dafermos \cite{MR785530}. The rarefaction wave fan is bounded by a characteristic line at least one side. However, it is possible that there exist several contacts of type D inside of the fan.
\end{remark}

\begin{property}
If a portion of the graph of the flux is added to the envelopes after a merging phenomenon, a centered rarefaction wave fan appears.
\end{property}

\subsection{Contact rarefaction}

A centered rarefaction wave fan of a fundamental solution is produced instantly at the moment of the initial time or a merging phenomenon. On the other hand, a contact rarefaction wave is produced continuously along a single sided contact shock connected with the maximum value $\bar \rho(t)$. A typical example can be found in Figures \ref{fig6}(a) and \ref{fig6}(b). The concave envelope in Figures \ref{fig6}(a) shows that it is about the moment that the genuine shock splits into two contacts. The concave envelope at a later time is given in Figure \ref{fig6}(b), which shows that the rarefaction region $(b_1,b_2)$ is expanding. This indicates that new information is being produced and propagates to the future. On the other hand the rarefaction region $(a_3,\bar \rho(t))$ from the convex envelope is shrinking. In other words the information from the past is destroyed if it meets this shock.

In summary a contact rarefaction wave is produced by the contact shock of single sided which connects the maximum from one side. For example, the right contact R in Figure \ref{fig6}(b) connecting $b_2$ and $\bar \rho(t)$ is the corresponding one. This single sided contact erases the information of the past from one side and produces new information from the other side.

\section{An example for a global picture}
\label{sect.full}

This final section is designed to provide a complete characteristic map that shows all the dynamics of shocks and rarefaction waves discussed before. We take a flux in Figure \ref{fig6} which is complicated enough for this purpose and satisfies (\ref{FluxForTrans}). Since the change of an envelope is linked to each stage of a solution, all the dynamics of a solution can be interpreted in terms of envelopes. First eight figures in Figure \ref{fig6} show the dynamics of the envelopes corresponding to the possible eight stages of the fundamental solution. As an example to show this connection we put an illustration of a signed fundamental solution in Figure \ref{fig6}(i), which belongs to the second stage, Figure \ref{fig6}(b). More examples of fundamental solution can be found in \cite[Figures 6--8]{HaKim}, which are actually obtained by computing the equation numerically. A complete characteristic map corresponding to this flux is given in Figure \ref{fig7} with stage numbers on the left. In the rest of this section we investigate the relation between the flux and its characteristic map.

%=================================================================
%==================================================================
\begin{figure}[htb]
\begin{minipage}[t]{3.7cm}
\centering \psfrag  {R}{R} \psfrag  {II}{D} \psfrag  {G}{G}
\psfrag  {L}{L} \psfrag  {f(u)}{\footnotesize $f(u)$}
\psfrag  {0}{\scriptsize $0$} \psfrag  {u}{\footnotesize $u$}
\psfrag  {a2-}{\small $a_2^-$} \psfrag  {a2+}{\small $a_2^+$}
\includegraphics[width=4cm]{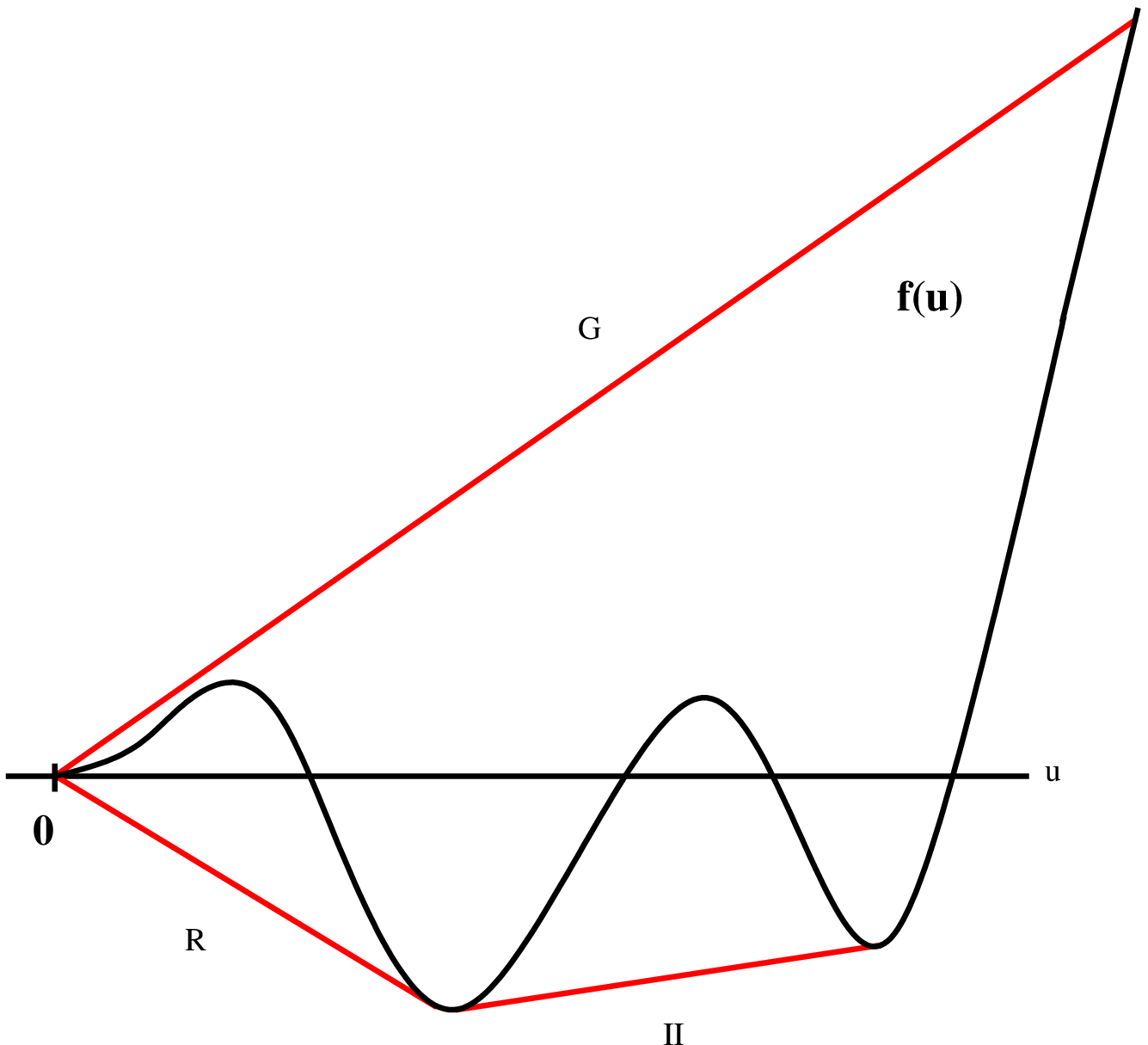}
(a) initial structure
\includegraphics[width=3.8cm]{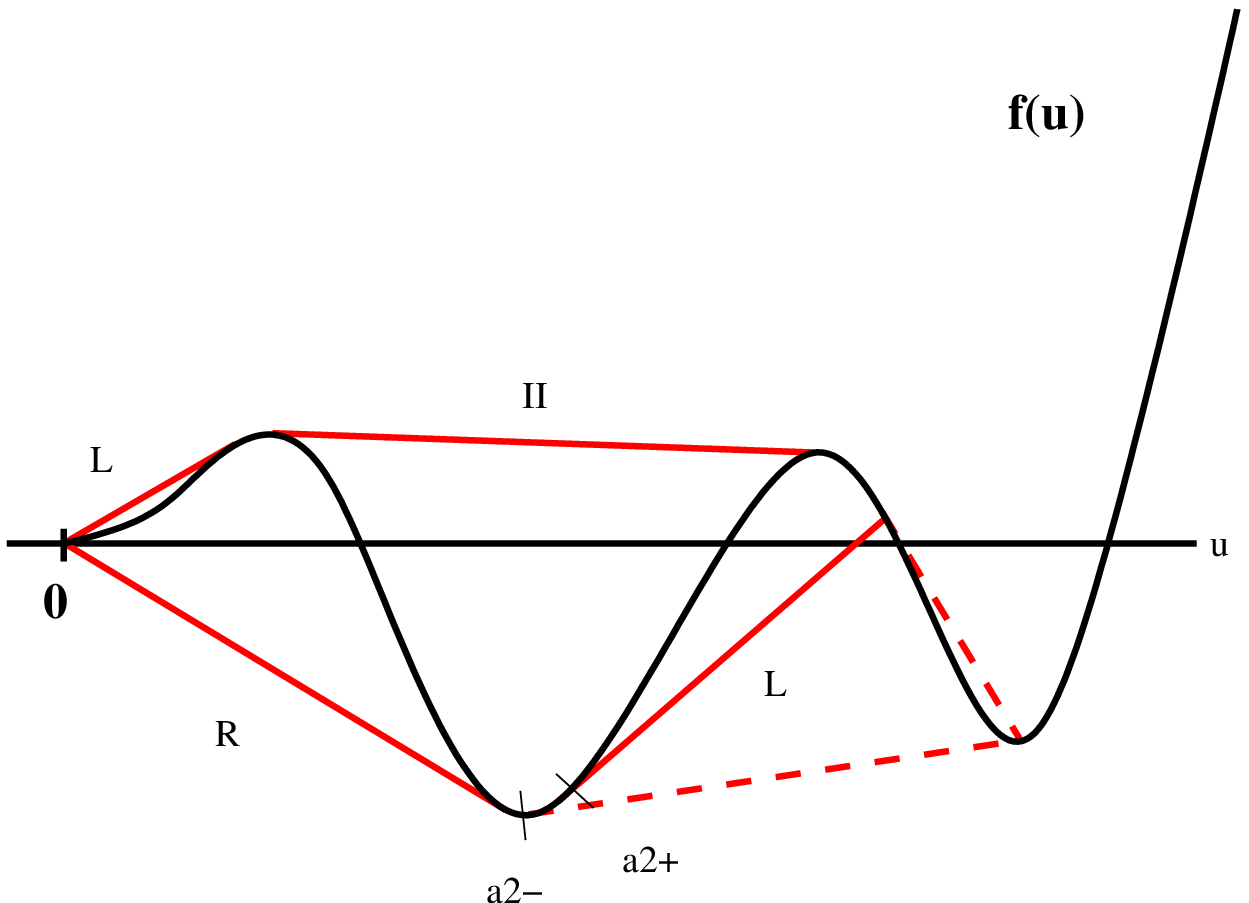}
(d) 1st merging
\\[1.05em]
\vskip .7cm
\includegraphics[width=3.8cm]{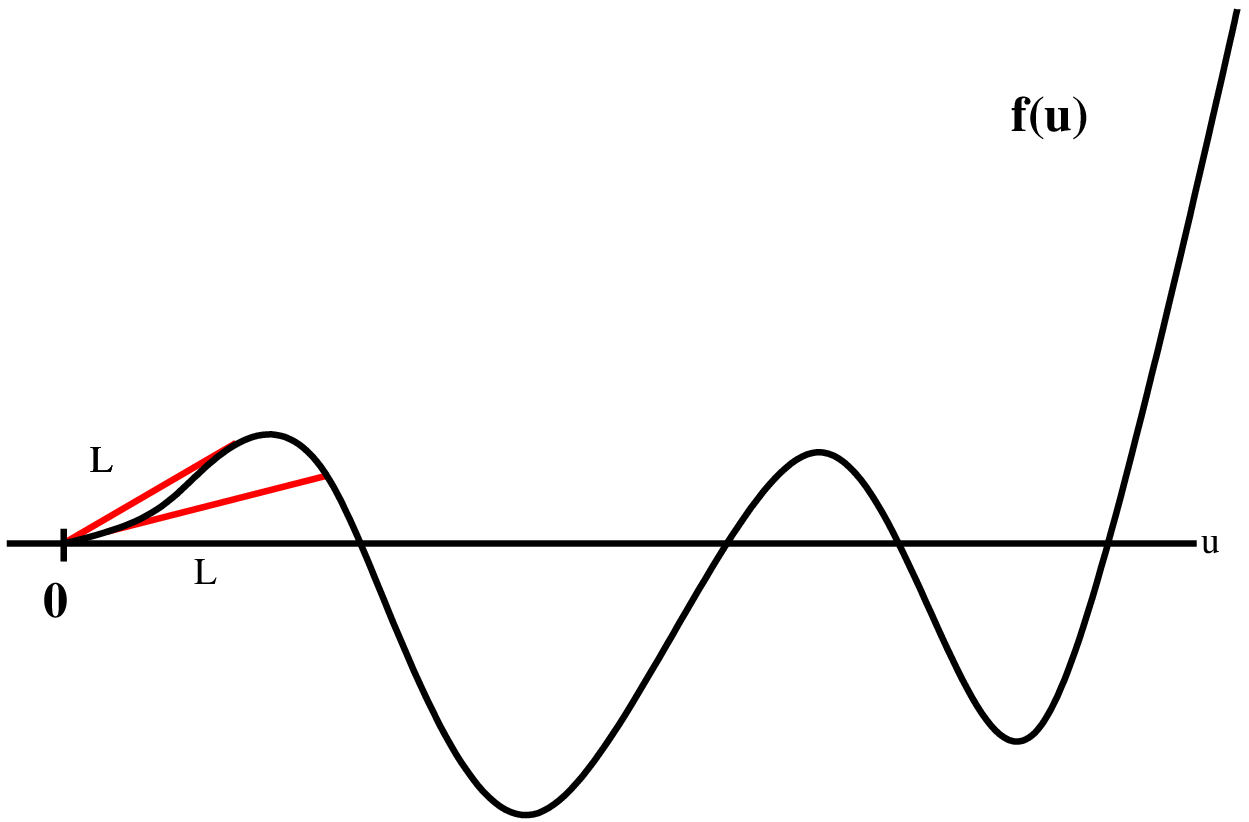}
%\\[0.5em]
(g) transforming
\end{minipage}
\hfill
\begin{minipage}[t]{3.7cm}
\centering \psfrag  {R}{R} \psfrag  {L}{L} \psfrag  {G}{G}
\psfrag  {II}{D} \psfrag  {f(u)}{\footnotesize $f(u)$}
\psfrag  {0}{\scriptsize $0$} \psfrag  {u}{\footnotesize $u$}
\psfrag  {a1}{\small $a_1$} \psfrag  {a2}{\small $a_2$}
\psfrag  {a3}{\small $a_3$} \psfrag  {a4}{\small $\bar \rho(t)$}
\psfrag  {b1}{\small $b_1$} \psfrag  {b2}{\small $b_2$}
\includegraphics[width=4cm]{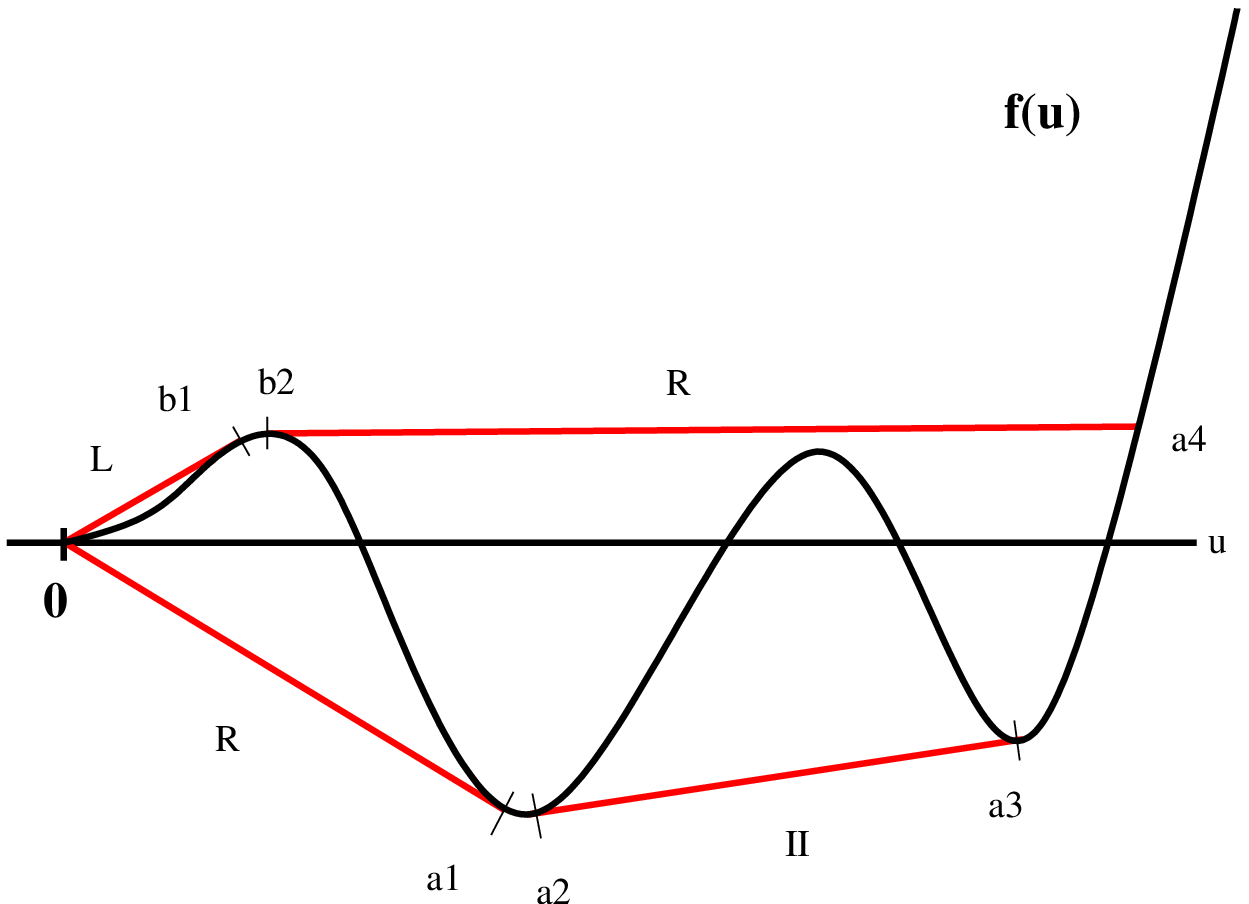}
(b) 1st branching
\\[0.6em]
\includegraphics[width=4cm]{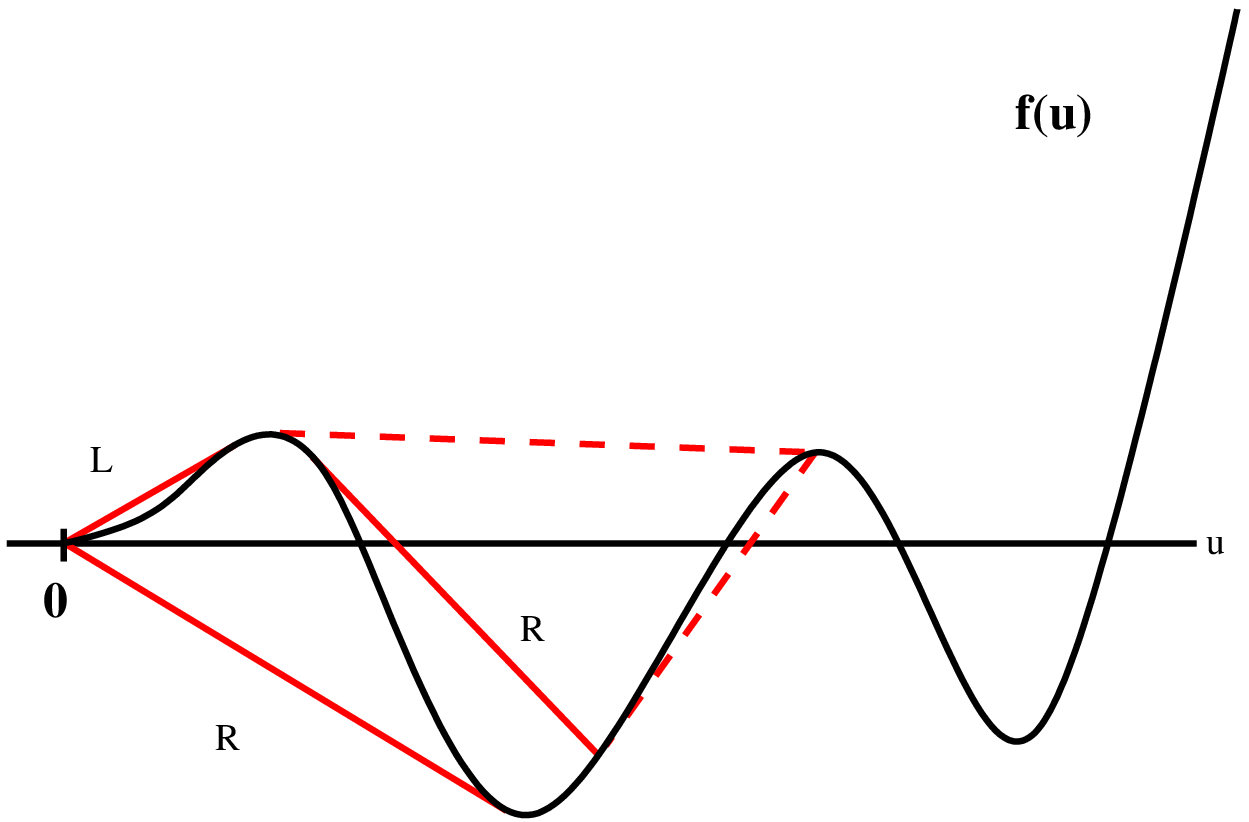}
(e) 2nd merging \vskip .5cm
\includegraphics[width=4cm]{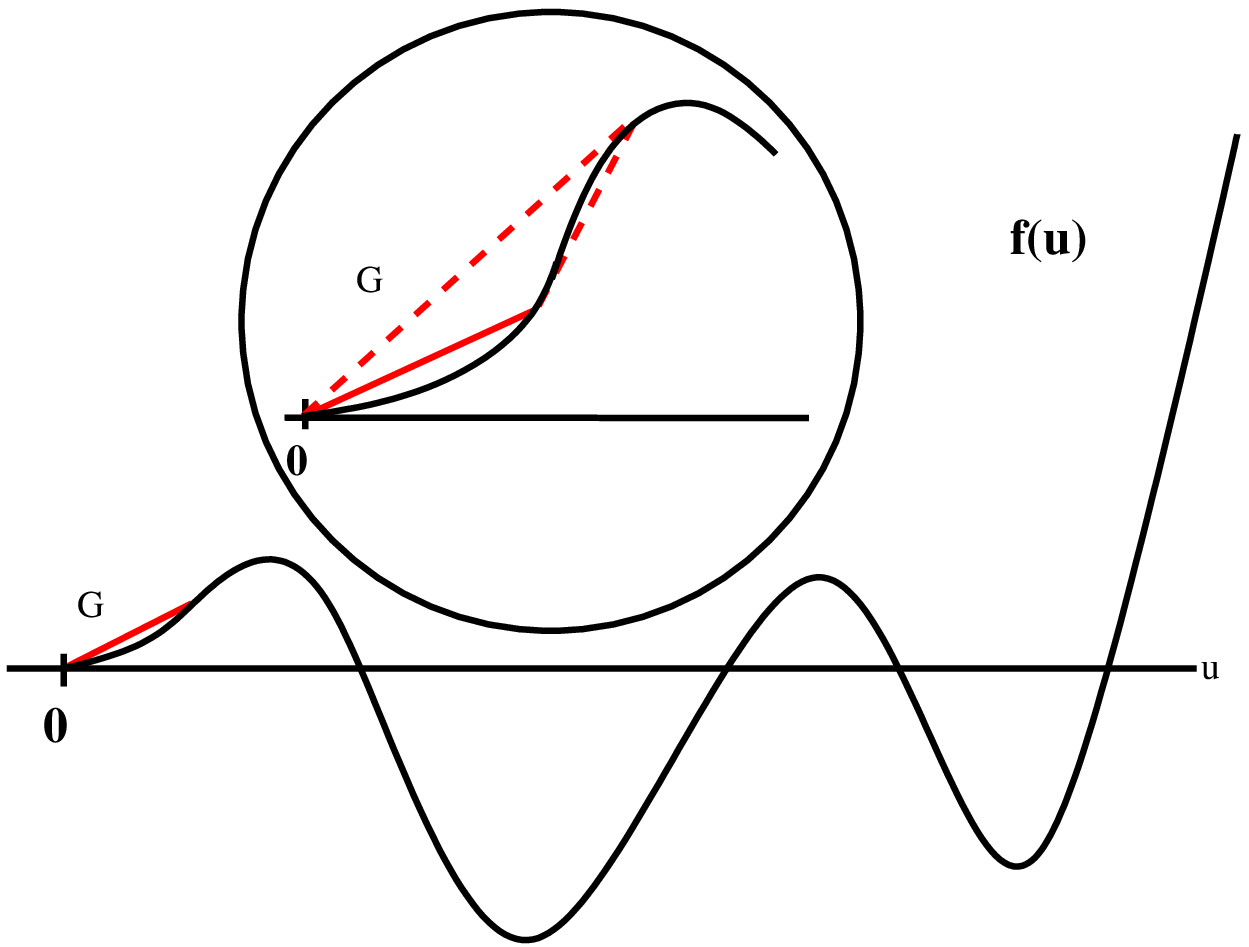}
%\\[0.4em]
(h) final merging
\end{minipage}
\hfill
\begin{minipage}[t]{3.5cm}
\centering \psfrag  {R}{\small R} \psfrag  {L}{\small L}
\psfrag  {G}{\small G} \psfrag  {II}{\small D} \psfrag  {a1}{\small $a_1$} \psfrag  {a2}{\small $a_2$} \psfrag  {a3}{\small $a_3$}
\psfrag  {a4}{\small $\bar \rho(t)$} \psfrag  {b1}{\small $b_1$}
\psfrag  {b2}{\small $b_2$} \psfrag  {f(u)}{\footnotesize $f(u)$}
\psfrag  {0}{\scriptsize $0$} \psfrag  {x}{\scriptsize $x$}
\psfrag  {u}{\footnotesize $u$}
\includegraphics[width=4cm]{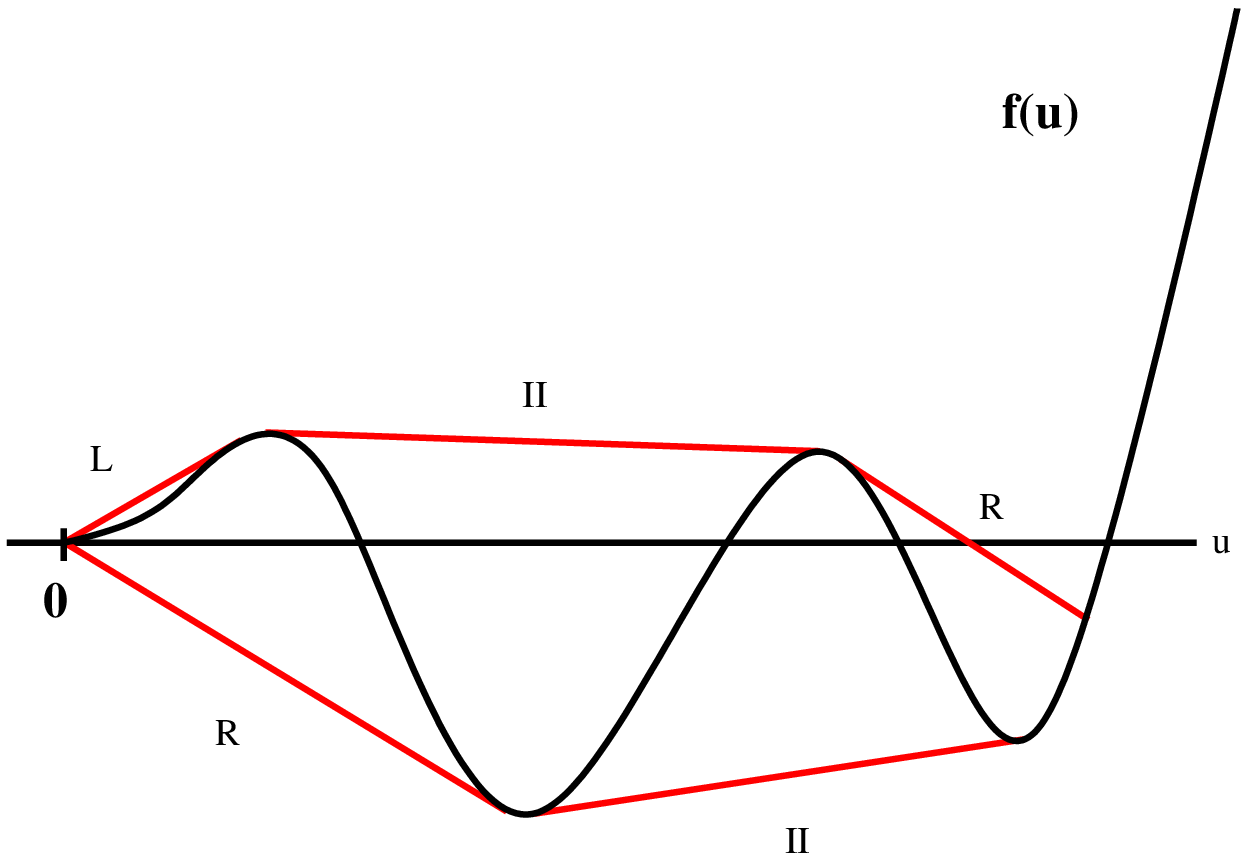}
(c) 2nd branching
\\[0.6em]
\includegraphics[width=4cm]{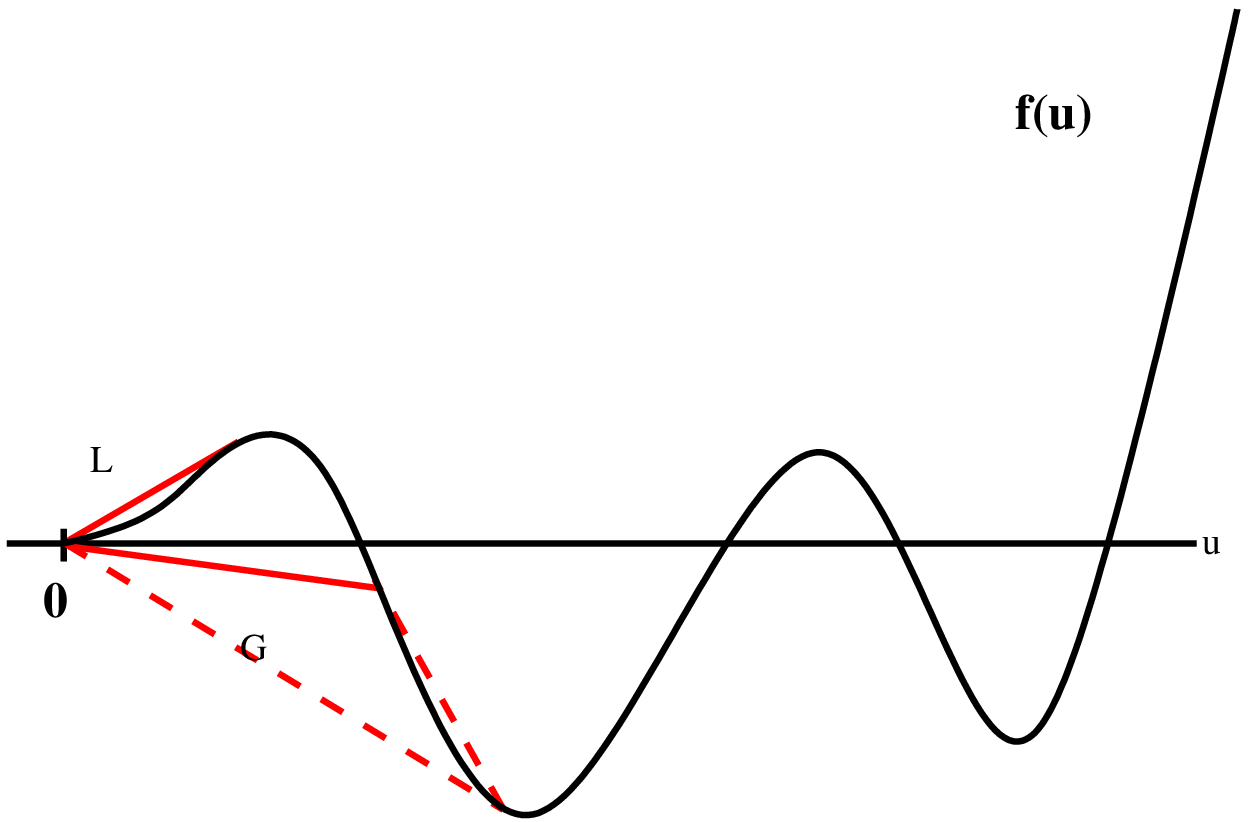}
(f) 3rd merging \vskip 0.1cm
\includegraphics[width=2.4cm]{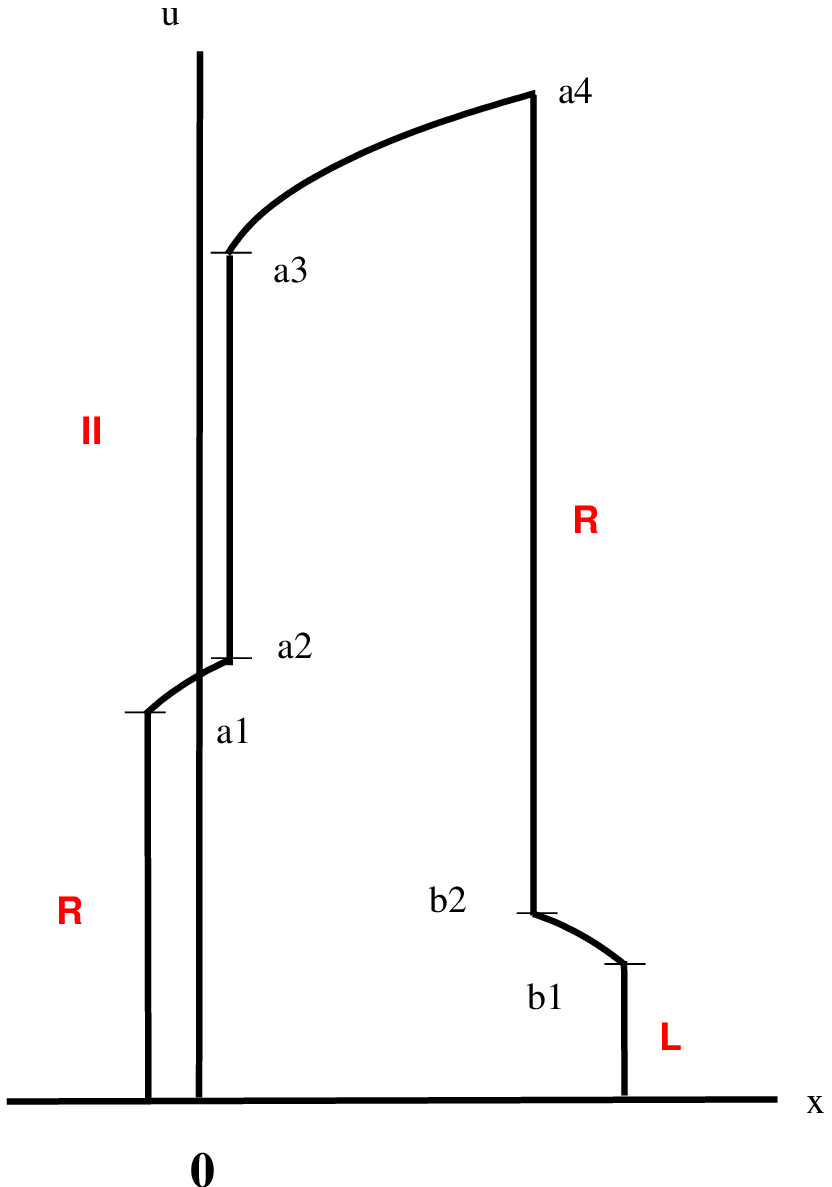}

(i) fundamental solution corresponding to (b)
\end{minipage}\caption{The envelopes of all eight stages. As the maximum of the fundamental solution decreases the corresponding envelopes change. G: genuine shock, R: right contact of single sided, L: left contact of single sided, D: contact of type D. $a_j, b_j$ in (b), (d) for horizontal coordinates but vertical in
(i).} \label{fig6}
\end{figure}

Due to the second hypothesis in (\ref{H1}), one can find a moment $t_0$
such that for all nonnegative $t<t_0$ the concave envelope consists of a single non-horizontal line. In the case there exists a decreasing genuine shock as denoted in Figure \ref{fig6}(a). In addition, there are two increasing shocks from the very beginning, a right and a double contacts. These contacts are connected by a wave fan centered at the origin. These right and double sided contacts move with constant speeds until Figure \ref{fig6}(d) arrives.

The slope of the concave line corresponding to the genuine shock decreases as time $t$ increases and hence the genuine shock curve is not a line as in Figure \ref{fig7}. We can observe two branching phenomena while the concave envelope moves toward \ref{fig6}(d), which are G$\to$L+R in \ref{fig6}(b) and  R$\to$D+R  in \ref{fig6}(c). Note that at every branching point all the curves and the line have the same slope as shown in Figure \ref{fig7}.

After that, we may observe three merging phenomena, D+R$\to$L in \ref{fig6}(d), L+D$\to$R \ref{fig6}(e), and  R+R$\to$G in \ref{fig6}(f). Notice that none of two contacts have the same slope after a merging phenomenon. In addition, centered wave fans appear after the first two merging phenomena. However, if a genuine shock is produced, any kind of rarefaction waves is not produced. This genuine shock moves to left slower and slower, and eventually stops.  Then, the genuine shock turns into a left contact which moves to the right, see Figures \ref{fig6}(g) and \ref{fig7}. Note that it is possible that, after the third merging, Figure \ref{fig6}(f), a left contact may appear instead of the genuine shock and directly go to the stage \ref{fig6}(g), which is the case discussed for Figure \ref{fig5}(b). Finally, this left contact meets the other left contact after a long time and generates a genuine shock as in Figure \ref{fig6}(h). The picture in the magnified circle shows the final merging process. From this moment this genuine shock is a unique discontinuity and persists forever.
%==================================================================
\begin{figure}\centering
\psfrag  {I}{{\bf I}} \psfrag  {II}{{\bf D}} \psfrag  {t}{$t$}
\psfrag  {x}{$x$} \psfrag  {R}{{\bf R}} \psfrag  {L}{{\bf L}}
\psfrag  {G}{{\bf G}} \psfrag  {0}{\small $0$}
\psfrag  {1}{\hskip -2mm(a)}\psfrag  {2}{\hskip -2mm(b)} \psfrag  {3}{\hskip -2mm(c)} \psfrag  {4}{\hskip -2mm(d)}
\psfrag  {5}{\hskip -2mm(e)}\psfrag  {6}{\hskip -1mm(f)} \psfrag  {7}{\hskip -2mm(g)} \psfrag  {8}{\hskip -2mm(h)}
\includegraphics[width=10cm,height =9.5cm]{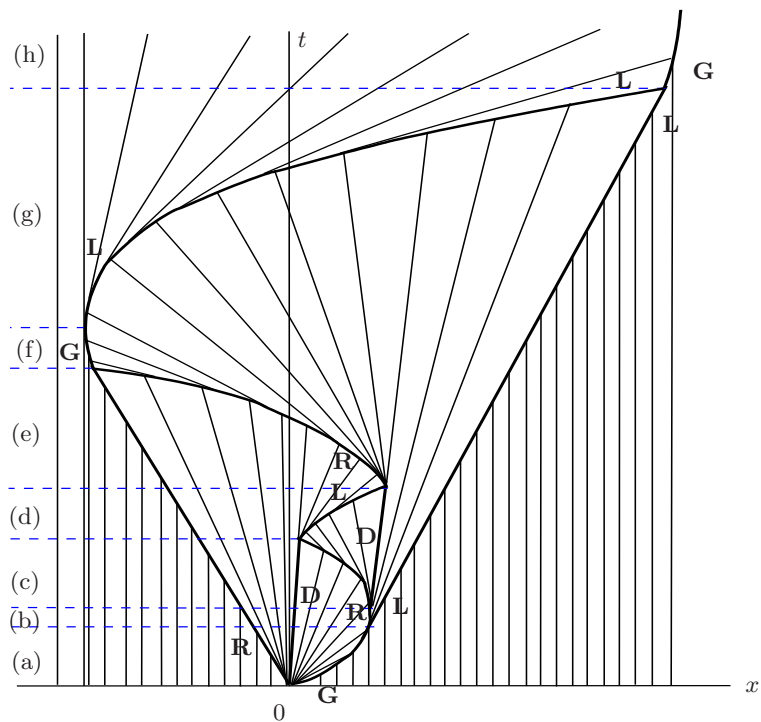}
\caption{Dynamics of characteristics. Letter in the left indicate the stage of each strip corresponding to stages in Figure \ref{fig6}.}\label{fig7}
\end{figure}
%==================================================================

%==================================================================
\begin{figure}[htb]\vskip -0.6cm\hskip -2cm
 \includegraphics[height=8cm]{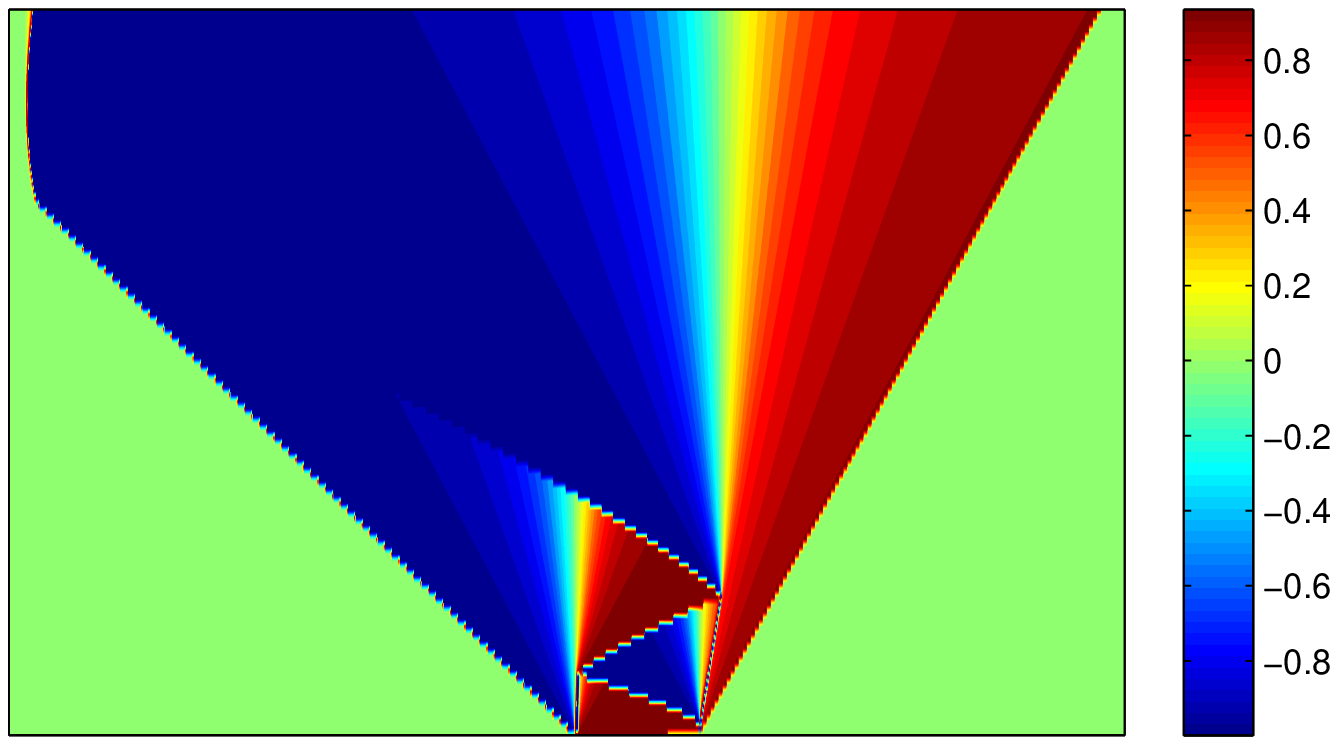}\vskip -0.9cm
  \hskip 5cm 0
\caption{[The horizontal axis is for space $x$ and the vertical axis is for time $t$.] The wave speed $f'(u)$ has been displayed in this figure. This figure clearly shows the interaction of shock waves and emerging centered rarefaction waves at the place where two shock waves are merged.}\label{fig8}
\end{figure}
%==================================================================

Note that Figure \ref{fig7} is for an illustration purpose and made under some exaggerations to keep the whole dynamics in a single figure. It seems interesting to compare this illustration with one obtained from an actual numerical solution. In fact, we have computed a fundamental solution numerically and then displayed its dynamics in Figure \ref{fig8}. In the figure we have displayed up to the beginning of Figure \ref{fig6}(g).

The wave speed $f'(u)$ is displayed in Figure \ref{fig8}, where $u$ is a numerically computed fundamental solution. One may observe the shock curves and easily distinguish if it is a left contact or a right contact. This shock curves match with Figure {\ref{fig7} pretty well except the ones near the initial time. For $t>0$ small the evolution of the solution is fast and one may observe numerically if the corresponding part is magnified which is omitted here. In this figure the appearance of the centered rarefaction waves and propagation of discontinuities discussed before are more clearly observed. To produce these figures we used the WENO method. These figures indicate that the theoretical explanation and the numerical simulation give a perfect match.

\begin{acknowledgement}
Authors would like to thank Jaywan Chung and Youngsoo Ha for useful discussions and their help in making figures in this paper.
\end{acknowledgement}

\def\cprime{$'$}
\providecommand{\bysame}{\leavevmode\hbox to3em{\hrulefill}\thinspace}
\providecommand{\MR}{\relax\ifhmode\unskip\space\fi MR }
\providecommand{\MRhref}[2]{%
  \href{http://www.ams.org/mathscinet-getitem?mr=#1}{#2}
}
\providecommand{\href}[2]{#2}

\end{document}